\documentclass[12pt]{article}

\usepackage{authblk}
\usepackage{amsfonts, amsmath, amssymb,amsthm}
\usepackage{amsfonts}
\usepackage{amssymb}
\usepackage{enumitem}
\usepackage[a4paper]{geometry}
\usepackage{mathtools}

\usepackage{algorithm}
\usepackage{algorithmic}

\usepackage{color}

\newtheorem{theorem}{Theorem}[section]

\newtheorem{lemma}[theorem]{Lemma}
\newtheorem{proposition}[theorem]{Proposition}

\theoremstyle{definition}
\newtheorem{example}[theorem]{Example}

\newtheorem{definition}[theorem]{Definition}

\newtheorem{remarks}[theorem]{Remarks}

\def\RR{\hbox{\sf I\kern-.14em\hbox{R}}}

\newcommand{\WW}{\omega}

\newcommand{\SG}{\mathfrak{S}}

\newcommand{\Rmax}{\RR_{\max}}
\newcommand{\Rmin}{\RR_{\min}}

\newcommand{\Rnnmax}{\Rmax^{n \times n}}

\newcommand{\fullchi}[1]{\bar{\chi}_{#1}}

\DeclareMathOperator{\maxperm}{\mathrm{maxperm}}

\numberwithin{equation}{section}

\newcommand{\maxzero}{-\infty}

\begin{document}

\baselineskip 5.8mm

\title{Computing the sequence of $k$-cardinality assignments}

\author{Amnon Rosenmann}
\affil{Institute of Discrete Mathematics \\
	Graz University of Technology, Graz, Austria \\
	rosenmann@math.tugraz.at}
\date{}
\maketitle

\begin{abstract}
\baselineskip 7mm
The $k$-cardinality assignment problem asks for finding a maximal (minimal) weight of a matching of cardinality $k$ in a weighted bipartite graph $K_{n,n}$, $k \leq n$.
The algorithm of Gassner and Klinz from 2010 for the parametric assignment problem computes in time $\mathcal{O}(n^3)$ the set of $k$-cardinality assignments for those integers $k \leq n$ which refer to ''essential'' terms of a corresponding maxpolynomial. 
We show here that one can extend this algorithm and compute in a second stage the other ''semi-essential'' terms in time $\mathcal{O}(n^2)$, which results in a time complexity of $\mathcal{O}(n^3)$ for the whole sequence of $k=1, \ldots, n$-cardinality assignments.
The more there are assignments left to be computed at the second stage the faster the two-stage algorithm runs. 
In general, however, there is no benefit for this two-stage algorithm on the existing algorithms, e.g. the simpler network flow algorithm based on the successive shortest path algorithm which also computes all the $k$-cardinality assignments in time $\mathcal{O}(n^3)$.
\end{abstract}
\noindent
\textbf{Keywords:} $k$-cardinality assignment problem; parametric assignment algorithm; max-plus algebra; full characteristic maxpolynomial
\section{Introduction}
The linear assignment problem is a basic problem in combinatorial optimization.
The time complexity for solving the $k$-cardinality assignment problem in a full  bipartite graph $K_{n,n}$ with weights on its edges is, for general $k$, $\mathcal{O}(n^3)$, for example through the Hungarian method.

Gassner and Klinz \cite{GK10} showed that one can compute the parametric assignment problem in time $\mathcal{O}(n^3)$.
It follows that the list of $k$-cardinality assignments that correspond to essential terms of the full characteristic maxpolynomial of the matrix representing the weights on $K_{n,n}$ can be computed in $\mathcal{O}(n^3)$.
We show here that one can, in fact, compute in time $\mathcal{O}(n^3)$ \textit{all} the $k$-cardinality assignments and not just the essential ones. 
The result is based on the full characteristic maxpolynomial having the property of being in full canonical form \cite{RLP19}, for which we give here a simple proof.

We note, however, that this time complexity does not improve on the existing algorithms.  The simpler network flow algorithm (see \cite{AMO93}), based on successive shortest paths, which augments one unit flow at each iteration, runs also in time complexity $\mathcal{O}(n^3)$.

In special cases where there are many assignments left to compute after the completion of the Gassner and Klinz algorithm then the fact that these left assignments can be computed in time $\mathcal{O}(n^2)$ may be of benefit.
\section{The $k$-linear assignment problem}
We are given a complete bipartite graph (biclique) $K_{n,n} = G(U,V;E)$ consisting of two disjoint sets of vertices $U = \{u_1, \ldots, u_n\}$ and $V = \{v_1, \ldots, v_n\}$, and an edge set $E$, adjoining each vertex $u_i \in U$ with each vertex $v_j \in V$ by an edge $e_{ij}$.
In addition, we are given a weight ("cost") function $w : E \to \Rmax = \RR \cup \{-\infty\}$ which assigns each edge $e_{ij}$ a weight $w_{ij}$.
The $k$-cardinality Linear Assignment Problem, or $k$-LAP (see e.g. \cite{DM97}, \cite{DLM01}, \cite{Vol04}, \cite{BJ16}), asks for finding a maximal (optimal) weight of a matching (independent edge set) of cardinality $k$, $k \leq n$, that is, a set $M_k \subseteq E$ of $k$ pairwise non-adjacent edges in $K_{n,n}$ which is of maximal total weight:  
\begin{equation}
\max_{\{i_1,\ldots,{i_k}\},\{j_1,\ldots,{j_k}\} \subseteq [1..n]} \; \Sigma_{l=1}^k w_{i_l j_l},
\end{equation}
where the indices $i_l$ as well as the indices $j_l$ are pairwise distinct.
\begin{remarks}
	\begin{enumerate}
		\item An analogous and similarly solved problem asks for computing the minimal assignment, and then one works over $\Rmin = \RR \cup \{+\infty\}$. 
		\item If the problem refers to sets $U$ and $V$ of different cardinalities $n$ and $m < n$, respectively, then we can always introduce $n-m$ spurious (dummy) vertices and $n(n-m)$ spurious edges of weights $-\infty$ ($+\infty$ in the minimal assignment case) so that $U$ and $V$ become of equal cardinality. 
		\item  We refer to \cite{BDM12} for a comprehensive treatment of linear and related assignment problems.
	\end{enumerate}
\end{remarks}

Representing the weight function as an $n \times n$ matrix $W = (w_{ij})$ over $\Rmax$, the $k$-LAP is about finding $k$ elements of $k$ different rows and $k$ different columns, such that their sum is maximal.
It can be formulated as an Integer Programming (IP) problem over the variables $x_{i,j}$, for $i,j = 1,\ldots,n$, as follows:
\begin{equation}
	\mbox{maximize } \, \sum_{i=1}^n \sum_{j=1}^n w_{ij} x_{ij}
	\label{eq:IP_goal}
\end{equation}
subject to the conditions
\begin{align}
	\label{eq:IP_cond}
	&\sum_{j=1}^n x_{ij} \leq 1 \qquad (i = 1, \ldots, n), \nonumber \\
	&\sum_{i=1}^n x_{ij} \leq 1 \qquad (j = 1, \ldots, n), \\
	&\sum_{i=1}^n \sum_{j=1}^n x_{ij} = k, \nonumber \\
	&x_{ij}	 \in \{0,1\} \qquad (i,j = 1, \ldots, n). \nonumber 
\end{align}
	
If we replace the last condition $x_{ij} \in \{0,1\}$ in \eqref{eq:IP_cond} by the inequalities
\begin{equation}
\label{eq:LAP}
x_{ij} \geq 0 \qquad (i,j = 1, \ldots, n),
\end{equation}
we obtain the corresponding Linear Programming (LP) problem.

A LAP can also be described in the language of network flows or in terms of matroids.
In what follows we omit the word "linear" and use the term $k$-assignment for $k$-linear (optimal) assignment.

\section{Max-plus algebra}
\label{sec:maxplusalgebra}
The values of the $k$-assignment problem, $k=1, \ldots, n$, appear as the coefficients of the full characteristic polynomial of the weight matrix $W$ in the setting of max-plus algebra.
In this section we present the relevant background about max-plus algebra. 

Max-plus algebra \cite{But10} is an algebra over the semifield $\Rmax = \RR \cup \{ -\infty \}$, equipped with the operations of addition $a\oplus b$ defined as $\max(a,b)$ and multiplication $a \odot b$ defined as $a+b$ in standard arithmetic, with the unit elements $-\infty$ for addition and $0$ for multiplication.
For the sake of readability, we suppress the multiplication sign $\odot$, writing $ab$ instead of $a \odot b$ and  $ax^3$ instead of $a \odot x^{\odot 3}$.
Also, when an indeterminate $x$ appears without a coefficient, as in $x^n$, then its coefficient is naturally the multiplicative identity element, i.e. $0$.

A \textbf{(formal) maxpolynomial} of degree $d$ over $\Rmax$ in the indeterminate $x$ is an expression of the form
\begin{equation}
\label{eq:defmaxpolynomial}
p (x)= \bigoplus _{k=0} ^d a_k x^k = \max \{a_k + k x: k=0, 1, \ldots, d \}
\end{equation}
with  $a_0 , \ldots, a_d \in  \Rmax$.
Each element of the generated algebra under the max and plus operations can be reduced to a unique expression of the form \eqref{eq:defmaxpolynomial} (where normally the maxmonomials with coefficient $-\infty$ are deleted) with $x$ treated as an indeterminate.

A maxpolynomial $p(x)$ induces also a function $\hat{p}(x)$ on $\Rmax$, which is convex and piecewise-affine.
Unlike the situation in standard arithmetic, two distinct formal maxpolynomials $p_1(x)$ and $p_2(x)$ may represent the same polynomial function, that is, $\hat{p}_1(x) = \hat{p}_2(x)$ as functions.
This happens because a maxmonomial $a_kx^k$ of $p(x)$ with the property that for every value of $x$ there exists another maxmonomial $a_lx^l$, $l \neq k$, of $p(x)$, such that $a_kx^k \leq a_lx^l$, does not contribute to the function $\hat{p}(x)$ and thus may be omitted.
The term $a_kx^k$ is said to be an \textbf{essential} term of $p(x)$ when, at some interval of $\RR$, $\hat{p}(x) = a_kx^k$ as functions. 
When $a_kx^k < \hat{p}(x)$ for all values of $x$ then $a_kx^k$ is said to be \textbf{inessential}.
Otherwise, when $\hat{p}(x) = a_kx^k$ at a single point, then we say that $a_kx^k$ is \textbf{semi-essential} (in the literature this term is also called inessential).

For each function $\hat{q}(x)$ there exists a unique maxpolynomial representation $p(x)$, such that every power $x^k$ of $p(x)$ appears with the maximal possible coefficient $a_k$, as long as the functional equality $\hat{p}(x)=\hat{q}(x)$ holds.
We then say that $p(x)$ is in \textbf{full canonical form} (in FCF) \cite{CGM80}.
Note that when $p(x)$ is in FCF than all its maxmonomials are either essential or semi-essential.
We adopt the convention $-\infty - (-\infty) = -\infty$.
\begin{proposition} [\cite{RLP19}]
	Let $p(x) = \bigoplus_{k=0}^{n} a_k x^k$ be a maxpolynomial of degree $n \geq 1$ and let $\lambda_1 \leq \lambda_2 \leq \cdots \leq \lambda_n$ be its roots.
	Then the following are equivalent characterizations of $p(x)$ to be in full canonical form.
	\begin{enumerate}
		\item $p(x) = a_n (x \oplus s_1) \cdots (x \oplus s_n)$ formally, for some $s_1, \ldots, s_n$.
		\item $p(x) = a_n (x \oplus \lambda_1) \cdots (x \oplus \lambda_n)$ formally.
		\item $\lambda_k = a_{k-1} - a_k$, $k = 1, \ldots, n$.
		\item Concavity: $a_k \geq (a_{k-1}+a_{k+1})/2$, $k = 1, \ldots, n-1$.
		\item $\hat{p}(\lambda_k) = a_k \lambda_k^k$ ($a_k + k \lambda_k$ in standard arithmetic) for $k = 1, \ldots, n$.
	\end{enumerate}
	\label{prop:fullcanonical}
\end{proposition}

The \textbf{max-plus roots} (\textbf{tropical} \textbf{roots}) of a maxpolynomial $p(x)$ are the points at which $\hat{p}(x)$ is non-differentiable.
The multiplicity of a root equals the change of the slope of $\hat{p}(x)$ at that root.
Equivalently, the roots of $p(x)$ are the values $\lambda \neq \maxzero$ of $x$ at which several maxmonomials $a_kx^k, a_{k+1}x^{k+1}, \ldots, a_{k+d}x^{k+d}$ in the corresponding full canonical form satisfy $a_k \lambda^k = a_{k+1}\lambda^{k+1} = \cdots = a_{k+d}\lambda^{k+d} = \hat{p}(\lambda)$ 
and the multiplicity of the root $\lambda$ is then $d$, the difference between the largest and the smallest index of these maxmonomials.
We also count $-\infty$ as a root with multiplicity $d$ whenever $a_0, a_1, \ldots , a_{d-1}$ are all equal to $\maxzero$ and $a_d \neq \maxzero$.
\subsection{The (full) characteristic maxpolynomial of a matrix}
\label{sec:charpoly}
Given an $n \times n$ matrix $A \in \Rnnmax$, the \textbf{max-plus permanent} of $A$ is
$$
	\maxperm(A) = \bigoplus_{\sigma \in \SG_n} a_{1 \sigma(1)} \cdots a_{n \sigma(n)},
$$
where $\SG_n$ is the group of permutations on $[n] =\{1, \ldots,n \}$.
The \textbf{characteristic maxpolynomial} of $A$, defined in \cite{CG83}, is
$$
	\chi_A (x) = \maxperm (x I \oplus A),
$$
where $I$ is the max-plus identity matrix with all entries on the main diagonal being 0 and all off-diagonal entries being $\maxzero$.
It is a maxpolynomial of degree $n$, say $\chi_A(x)=\bigoplus_{k=0}^n a_k x^{n-k}$, with $a_k$ being the maximal value among all principal maxpermanent minors of order $k$ of $A$.
The (tropical) roots of the characteristic maxpolynomial $\chi_A (x)$ are called the \textbf{max-plus eigenvalues} of $A$.
As shown in \cite{ABG01}, they can be asymptotically computed from the eigenvalues of an associated parametrized standard matrix with exponential entries.

Another maxpolynomial related to $A$ is its \textbf{full characteristic maxpolynomial}, defined as
$$
	\fullchi{A} (x) = \maxperm (x 0 \oplus A),
$$
where $0$ is the matrix with zeros in all its entries.
The difference between $\fullchi{A} (x)$ and $\chi_A (x)$ is that in $\fullchi{A} (x)$ the indeterminate $x$ appears in all entries instead of just on the diagonal.
The coefficients $a_k$ of $\fullchi{A}(x)=\bigoplus_{k=0}^n a_k x^{n-k}$ equal the maximal value among all maxpermanent minors (not just principal maxpermanent minors) of order $k$ of $A$.
That is, $a_k$ equals the value of a $k$-assignment with weight matrix $A$.
The (tropical) roots of the full characteristic maxpolynomial $\fullchi{A} (x)$ are called the \textbf{max-plus singular values} of $A$.
Similar to the max-plus eigenvalues, the max-plus singular values can be asymptotically computed from standard singular values (see \cite{Hook14} and \cite{DeSDeM02}).

An essential (formal) property of the full characteristic maxpolynomial, which distinguishes it from the characteristic maxpolynomial, is that it is in full canonical form.
This property allows us to compute efficiently all the $k$-assignments. 
\begin{proposition} [see e.g. \cite{RLP19}]
	Let $A \in \Rnnmax$. Then the full characteristic maxpolynomial $\fullchi{A} (x)$ is in FCF.
	\label{pr:fullchar_is_FCF}
\end{proposition}
Indeed, the coefficients of $\fullchi{A}(x)=\bigoplus_{k=0}^n a_k x^{n-k}$ are the solutions of the IP problem \eqref{eq:IP_goal}, \eqref{eq:IP_cond}.
If we look at the corresponding LP problem \eqref{eq:IP_goal} under the conditions \eqref{eq:IP_cond}, with \eqref{eq:LAP} replacing the last condition of \eqref{eq:IP_cond}, then it is clear that the set of solutions $\WW_k$ to the $k$-LP problems form a concave set:
$\WW_k \geq (\WW_{k-1}+\WW_{k+1})/2$,
which is one of the equivalent conditions appearing in Proposition~\ref{prop:fullcanonical}.
The proof would be complete if these solutions were integers.
However, since the standard form of the constraint matrix of our LP problem is a $(2n+1)\times(n^2+2n)$ \emph{totally unimodular} matrix \cite{DM97} then the solutions are indeed integers: $\WW_k = a_k$.

\section{The sequence of $k$-assignments is concave}
We give here a direct proof of the FCF property of the full characteristic polynomial without referring to linear programming or network flow but rather demonstrate it on the bipartite graph.
The method of proof will serve us in what will follow.

First, we show the known fact (see \cite{DM97}) that it is always possible to form a (maximal) $(k+1)$-assignment from the set of vertices matched by a $k$-assignment and an additional pair of vertices.
Given a set $E'$ of matched pairs in a bipartite graph $G(U,V;E)$, an \textbf{alternating path} in $E$ with respect to $E'$ is a path of positive length (may also be of length 1)
that alternately switches between edges of $E'$ and those of $E \setminus E'$. 
If an alternating path starts and ends in a vertex of $E \setminus E'$ then it is called \textbf{augmenting}.
\begin{lemma}
	Let $K_{n,n} = G(U,V;E)$ be a complete bipartite weighted graph on vertices $U=\{u_1,\ldots,u_n\}$ and $V=\{v_1,\ldots,v_n\}$.
	Let $M_k$, $k=1,\ldots,n$ be a sequence of $k$-assignments on $K_{n,n}$.
	Then, after possibly renaming the vertices, the set of matched vertices in $M_k$ may be chosen to be $\{u_1,\ldots,u_k, v_1,\ldots,v_k\}$, for $k=1,\ldots,n$.
	\label{lem:k-assign}
\end{lemma}
\begin{proof}
	Let $E_k, E_{k+1}$ be the set of edges of $M_k, M_{k+1}$, respectively.
	The edges of $E_k \cup E_{k+1}$ form a disjoint (with no common vertices) union of alternating paths.
	Because $|E_{k+1}| > |E_k|$, at least one these paths $p$ is augmenting: it starts in a vertex $u_{k+1} \in M_{k+1} \setminus M_k$, ends in a vertex $v_{k+1} \in M_{k+1} \setminus M_k$, and all the inner vertices are in $M_k \cap M_{k+1}$.
	By the maximality of total weight of the two assignments, we may take the rest of the matches of $M_k$ to be also in $M_{k+1}$, that is, $M_{k+1} = M_k \triangle p$, the symmetric difference of $M_k$ and $p$.
	Hence, $M_{k+1}$ contains exactly one pair of vertices that is not in $M_k$.
\end{proof}
In terms of matrices, Lemma~\ref{lem:k-assign} says that given a matrix $W \in \Rmax ^{n \times n}$, there exit permutation matrices $P,Q \in \RR^{n \times n}$, such that the $k$-assignments in $W'=PWQ$ occur on its leading principal $k \times k$ submatrices, for $k=1,\ldots,n$.

We give now an alternative proof of Proposition~\ref{pr:fullchar_is_FCF} in the language of $k$-assignments on bipartite graphs.
We define a $0$-assignment to be $M_0 = \emptyset$ with weight $\WW_0=-\infty$.
\begin{proposition}
	Given a complete bipartite weighted graph $K_{n,n}$, $n \geq 2$, the sequence of weights $\WW_k$, $k=0,\ldots,n$, of its $k$-assignments is concave. 
	\label{pr:concave}
\end{proposition}
\begin{proof}
	Let $M_{k-1}$, $M_k$ and $M_{k+1}$ be $3$ consecutive assignments with weights $\WW_{k-1}$, $\WW_k$ and $\WW_{k+1}$, respectively.
	For the concavity property, we need to show that $\WW_k \geq (\WW_{k-1}+\WW_{k+1})/2$.
	By Lemma~\ref{lem:k-assign}, we may assume that $M_{k+1}$ matches the vertices matched in $M_{k-1}$ and additional $4$ vertices and it consists of some of the pairs matched in $M_{k-1}$ and other matched pairs which form two augmenting paths $p_1,p_2$ with the rest of the edges of $M_{k-1}$.
	That is,
	$$
	p_1 \cap p_2 = \emptyset
	$$
	and
	$$
		p_1 \cup p_2 = M_{k+1} \triangle M_{k-1}.
	$$
	For $i=1,2$, let
	$$
		g_i = w(p_i \cap M_{k+1}) - w(p_i \cap M_{k-1})
	$$
	be the gain exhibited by $p_i$, that is, the difference between the total weight $w$ of the pairs of $M_{k+1}$ and those of $M_{k-1}$ in $p_i$.
	Then 
	$$
		\WW_{k+1} = \WW_{k-1} + g_1 + g_2.
	$$	
	Suppose, without loss of generality, that $g_1 \geq g_2$.
	Then we form a (not necessarily maximal) matching $M'_k$,
	$$
		M'_k = M_{k-1} \triangle p_1
	$$
	and we have
	$$
		g_1 = w(M'_k) - \WW_{k-1}.
	$$
	By the maximality of $\WW_k$,
	$$
		\WW_k \geq w(M'_k) \geq (\WW_{k-1}+\WW_{k+1})/2
	$$
	and the proof is complete.
\end{proof}

\begin{example}
	Let 
	$$
	A = \begin{bmatrix}
	-\infty & 8 & 5 & 0 \\
	10  & 8 & 5 & -\infty \\
	8 & 0 & 5 & 4 \\
	5 & 4 & -\infty & -\infty
	\end{bmatrix}
	$$
	be the matrix representation of a weighted bipartite graph $K_{4,4}$ with disjoint sets $U$ (rows) and $V$ (columns), that is, the weight of the edge $(u_i,v_j)$  is $A(i,j)$.
	The $k$-assignments, $1 \leq k \leq 4$, are: $M_1 = \{(u_2,v_1)\}$ of weight $\WW_1=10$, $M_2 = \{(u_2,v_1)$, $(u_1, v_2)\}$ of weight $\WW_2=10+8=18$, $M_3 = \{(u_2,v_1)$, $(u_1, v_2)$, $(u_3,v_3)\}$ of weight $\WW_3=10+8+5=23$ and  $M_4 = \{(u_2,v_1)$, $(u_1, v_3)$, $(u_3,v_4)$, $(u_4, v_2)\}$ of total weight $\WW_4=10+5+4+4=23$.
	The sequence of weights $(10,18,23,23)$ is clearly concave and it forms the coefficients of the full characteristic maxpolynomial $\fullchi{A} (x) = x^4 \oplus 10x^3 \oplus 18x^2 \oplus 23x \oplus 23$. 
\end{example}

\section{Computing the $k$-assignments}
In \cite{GK10} Gassner and Klinz presented an algorithm which computes the \textit{essential} terms $a_k x^{n-k}$ of the characteristic maxpolynomial $\chi_A(x)$ of a matrix $A \in \Rnnmax$, including the corresponding \textit{principal} $k$-assignments.
The same algorithm works also for the full characteristic maxpolynomial $\fullchi{A} (x)$, and as Hook showed \cite{Hoo16}, it can be further extended to the computation of the maxpermanent of a matrix whose entries are maxpolynomials.
Since the full characteristic maxpolynomial is in FCF then computing its semi-essential terms after having its essential ones is immediate.
We will show that it is also possible to efficiently compute \textit{all} the $k$-assignments after having the ones obtained by the algorithm of Gassner and Klinz.
\subsection{Computing the $k$-assignments corresponding to essential terms}
For the sake of completeness, we describe here the algorithm of Gassner and Klinz (the G-K algorithm) for the computation of the $k$-assignments that refer to the essential terms of the full characteristic maxpolynomial.
The major steps of the algorithm are described in Algorithm~\ref{alg:essential}.
\begin{algorithm}
	\caption{Computing the essential $k$-assignments}
	\begin{algorithmic}
		\STATE \textbf{Input:} $W$ \COMMENT{weight matrix} 
		\STATE \textbf{Output:} $out[k]$ \COMMENT{list of assignments $M_k$ referring to essential terms of $\fullchi{W} (x)$}; \\
		$(\lambda_j,mult_j)$ \COMMENT{list of the max-plus singular values of $W$ including multiplicities}
		\STATE \textbf{Initialization:} \\
		\quad $M_0 (x) \leftarrow \{e^{x}_{ii} \, | \, i=1,\ldots,n \}$ \COMMENT{initial assignment}; \\
		\quad $\overleftarrow{G}(x) \leftarrow$ initial residual graph; $T \leftarrow$ initial longest path tree; \\
		\quad $\{\kappa(q)\} \leftarrow$ \mbox{initial vertex keys}; \\
		\quad  $j \leftarrow 0$; $k \leftarrow 0$; $b \leftarrow \infty$ 
		\WHILE{$(k<n) \wedge (b > -\infty)$}
		\STATE $b \leftarrow \mbox{new maximal vertex key}$
		\STATE $e \leftarrow \mbox{new pivot edge}$
		\STATE $p \leftarrow \mbox{new pivot vertex}$
		\STATE Update $T$ with respect to $e$ and $p$
		\IF{$T$ contains a cycle $C(x)$}
		\STATE $j \leftarrow j+1$
		\STATE $d \leftarrow -w^x(C(x))$ 
		\STATE $(\lambda_j,mult_j) \leftarrow (b,d)$ \COMMENT{(new max-plus singular value of $W$, its multiplicity)}
		\STATE $k \leftarrow k+d$
		\STATE $M_k (x) \leftarrow M_k (x) \triangle C(x) $ \COMMENT{new assignment}
		\STATE $out[k] \leftarrow M_k$ \COMMENT{constant edges of new assignment}
		\STATE Update $\overleftarrow{G}_k$, $T$, $\{\kappa(q)\}$
		\ELSE
		\STATE Update $\{\kappa(q)\}$
		\ENDIF
		\ENDWHILE
		\IF{$k<n$}
		\STATE $j \leftarrow j+1$
		\STATE $d \leftarrow n-k$ 
		\STATE $(\lambda_j,mult_j) \leftarrow (-\infty,d)$ \COMMENT{(new max-plus singular value of $W$, its multiplicity)} 
		\ENDIF
	\end{algorithmic}
	\label{alg:essential}
\end{algorithm}

Given a weight (cost) matrix $W=(w_{ij})$ of size $n \times n$, we form the directed bipartite graph $K_{n,n} = G(U,V;E)$ enhanced with weights $w_{ij}$ applied to the "constant edges" $e_{ij}$ from the vertices $u_i$ to the vertices $v_j$.
Then we add $n^2$ "parametric edges" $e^{x}_{ij}$ directed from $u_i$ to $v_j$, for $i,j=1,\ldots,n$, each edge of weight $x$.
Finally, we add a root $r$ and $n$ edges $e^r_i$, each of weight $0$, that start in $r$ and terminate in $u_i$, $i=1,\ldots,n$.
These edges remain fixed throughout the algorithm, whereas the other edges may change direction, accompanied by negation of their weight.
The obtained graph is denoted $G(x)$.
   
As in Hook's presentation of the algorithm \cite{Hoo16}, we begin with a value $b_0$ which is larger than the maximal entry of $W$.
Clearly, for any value of $x \in [b_0,\infty)$, an assignment (matching of cardinality $n$ of maximal weight) in $G(x)$ contains $n$ pairwise non-adjacent parametric edges, so that the weight of the assignment is $\WW_0 (x) = nx$.
The weight of this assignment corresponds to the term $x^n$ of the full characteristic maxpolynomial $\fullchi{W} (x) = \maxperm (x 0 \oplus W)$. 
Then the value of the parameter $x$ is gradually decreased and the constant entries of $W$ are "exposed".
This is done in a bounded number of discrete steps, where at each step a new value $b_i \leq b_{i-1}$ is computed and several updates are performed.
If during these updates some special condition is fulfilled (the emerging of a cycle after updating the spanning tree, see below) then $b_i$ is the value of the next max-plus singular value $\lambda_j$ and a new assignment is computed.
The new assignment has less parametric edges than the previous one, and the difference is the multiplicity of $\lambda_j$ (although, in fact, the singular values $\lambda_j$ need not be distinct and the actual multiplicity of $\lambda_j$ is then the sum of the corresponding multiplicities).
The weight $\WW_k (x) = a_k + (n-k)x$, $a_k \in \Rmax$, of the new assignment that applies to the current value of $x$ corresponds to a term $a_k x^{n-k}$ of $\fullchi{W}(x)$, that is, $\WW_k = a_k$ is the total weight of a $k$-assignment in $G(U,V;E)$.

The question now is how do we compute the essential $k$-assignments? 
The G-K algorithm proceeds as follows.
First we let the initial assignment $M_0 (x)$ for $b_0 \leq x < \infty$ to consist of the parametric edges $e^{x}_{ii}$, $i=1,\ldots,n$.
Then we construct the \textbf{residual graph} $\overleftarrow{G}(x)$.
The residual graph is obtained from the graph $G(x)$ by reversing the direction of the matching edges $e^{x}_{ii}$ and negating their weights.
A main tool of the algorithm consists of constructing and maintaining in the residual graph a parametric longest path tree $T$
according to the algorithm of Young, Tarjan and Orlin \cite{YTO91} (in \cite{YTO91} the constructed tree solves the analogous \textbf{parametric shortest path} problem).
This is a spanning tree with the property that the sum of weights of a path from the root to any vertex gives the longest existing path that reaches this vertex.
At the beginning the tree $T$ may consist e.g. of the edges $e^{r}_{i}$, for $i = 1,\ldots,n$, $e^{x}_{1j}$, for $j=2,\ldots,n$, and $e^{x}_{21}$, which gives a path of length $0$ to each of the vertices $u_i$ and a path of length $x$ to each of the vertices $v_j$.

We notice that the construction of a longest path tree $T$ is only possible when $\overleftarrow{G}(x)$ does not contain (reachable) cycles of positive weight: once such a cycle exists then it can be taken unlimited number of times while increasing the weight of the path indefinitely.
This property is crucial: when the value of $x$ drops below some threshold $b_i$ then the graph $G(x)$ contains a matching of $n$ pairs of vertices whose total weight is larger than the current assignment (if such a matching does not exist the algorithm terminates).
When $x < b_i$ the construction of a longest path tree is not possible anymore and this is demonstrated in the algorithm in the emergence of a cycle of positive weight when the tree is updated (see \cite{AMO93} for this property in terms of network flow). At that point we get a new max-plus singular value $\lambda_j = b_i$.

In general, as the value of $x$ decreases then $T$ needs to be updated more often than updating the assignment.
The update of $T$ is performed by trying to replace an edge $e=p \to q$ by an edge $e'=p' \to q$ (a \textbf{pivot edge}).
If the sub-tree of $T$ rooted at $q$ does not contain the vertex $p'$ then the replacement succeeds and the new tree is well-defined.
After updating the vertex keys (as explained below) the algorithm continues by looking for the next pivot edge.
Otherwise, the replacement of $e$ by $e'$ results in a cycle $C(x) = q \to q_1 \to \ldots \to q_{t-1} \to p' \to q$ of positive weight.
In this case we do the following.
The cycle $C(x)$ is alternating between edges of the current assignment $M_k (x)$ and non-matching edges.
We change the direction of the edges of $C(x)$ and negate their weights (thus turning the cycle to be of negative weight) while obtaining a new residual graph $\overleftarrow{G}(x)$.
If the current assignment is $M_k (x)$ then the new assignment is $M_{k+d} (x) = M_k (x) \triangle C(x) $, the  symmetric difference of $M_k(x)$ and $C(x)$. 
Here $k+d$ is the number of constant edges in the new assignment and $d=-w^x(C(x))>0$, where $w^x(C(x))$ denotes the integral coefficient of $x$ in the total weight of the edges of $C(x)$.
$d$ is also the multiplicity of the corresponding max-plus singular value.
Finally, we update the longest path tree: the edges of the new tree go in the direction $q \to p' \to q_{t-1} \to \ldots \to q_1$.

It remains to show how to find the pivot edge $e'$ with which we update the longest path tree. 
To reach this goal \cite{GK10} apply the algorithm of \cite{YTO91}.
For each directed edge $e$ of the residual graph $\overleftarrow{G}_k (x)$ let its weight be written as $w(e)=w^c(e)+w^x(e)x$, with $w^c(e) \in \Rmax$ is the constant part and $w^x(e) \in \{0,-1,1\}$ is the parametric part.
Then let $w(q)=w^c(q)+w^x(q)x$, where $w^c(q) \in \Rmax$ and $w^x(q) \in [-n \, .. \, n]$, be the weight of a vertex $q$ of $\overleftarrow{G}_k (x)$, defined as the sum of the weights along the path from the root to $q$ in the current tree.
For each directed edge $e = p \to q$ of the residual graph $\overleftarrow{G}$ its \textbf{key} is defined to be 
\begin{equation}
\kappa(e) = \frac{(w^c(p)+w^c(e))-w^c(q)}{w^x(q)-(w^x(p)+w^x(e))},
\label{eq:edgekey}
\end{equation}
where $\kappa(e)$ is defined to be $-\infty$ if the denominator in \eqref{eq:edgekey} is not positive.
Note that $\kappa(e)$
gives the value of $x$ at which the weight of the alternative path to $q$ through $e$ becomes equal to the weight of the path to $q$ along the current tree, and for $x < \kappa(e)$ the alternative path is the longer one.
The key of a vertex $q$ is then defined to be the maximum over the keys of all edges that terminate in $q$:
\begin{equation}
\kappa(q) = \max_p \{\kappa(e) \, | \, e=p \to q\}.
\label{eq:vertexkey}
\end{equation}
The update of the tree $T$ is done by choosing a vertex $q$ of $\overleftarrow{G}_k (x)$ of maximal key (a \textbf{pivot vertex}) as well as an edge $e=p \to q$ with $\kappa(e) = \kappa(q)$ (a pivot edge) and performing the update as described above.
Then we need to update the keys of all the vertices in the sub-tree with root $q$ before continuing.

For more details we refer to \cite{GK10}, \cite{Hoo16} and \cite{YTO91}.
\subsection{Computing the $k$-assignments corresponding to semi-essential terms}
Algorithm~\ref{alg:essential} computes the $k$-assignments which correspond to all the essential terms of the full characteristic maxpolynomial $\fullchi{W}(x)$ and some which correspond to semi-essential terms.
Since $\fullchi{W}(x)$ is in FCF we can easily deduce the values of the missing $k$-assignments.
The question is how to compute the missing $k$-assignments themselves. 
In order to compute all the $k$-assignments in Algorithm~\ref{alg:essential}, each newly computed max-plus singular value $\lambda_j$ should come with multiplicity $1$ (we remind that this is not necessarily its actual multiplicity since the computed singular values are not necessarily distinct).
The G-K algorithm does not address the case of semi-essential terms, and it can be shown that in the presence of multiple pivot edges at the same time, the choices we make influence the multiplicities of the newly computed singular values.  
It is not clear whether it is always possible to choose the pivot edges in Algorithm~\ref{alg:essential} in such a way that the singular values have multiplicity $1$, and if yes - if it can be done in an efficient way.
Nor isn't it clear whether it is possible to direct the algorithm in such a way that the max-plus singular values will occur in full multiplicity, that is, that each computed new singular value will be different from the previous one, although one can try to choose the pivot edges in a way that serves this or the other purpose.

However, we will show that after obtaining the $k$-assignments from Algorithm~\ref{alg:essential}, the missing assignments, those that correspond to semi-essential terms of $\fullchi{W}(x)$, are easy to compute.
\begin{definition}
	The terms $a_i x^i$ and $a_j x^j$ of a maxpolynomial $p(x) = \bigoplus _{k=0}^n a_k x^k$ are \textbf{adjacent} if there exists $c \in \RR$ such that $a_i c^i = a_j c^j = \hat{p}(c)$. 
\end{definition}
\begin{theorem}
	Let $W \in \Rnnmax$ and suppose we are given a $k$-assignment $M_k$ and a $(k+d)$-assignment $M_{k+d}$ which refer to adjacent terms of $\fullchi{W} (x)$.
	Then one can extract out of these assignments a set of in-between assignments $M_{k+i}$, $i=1,\ldots,d-1$.
	It follows that one can compute the assignments that refer to semi-essential terms when given those that refer to essential terms. 
	\label{th:adjacent}
\end{theorem}
\begin{proof}
	As in the proof of Proposition~\ref{pr:concave}, we start with the assignment $M_k$ and transform it to the assignment $M_{k+d}$ in $d$ steps.
	The set of matched pairs in $M_k \cup M_{k+d}$ form disjoint paths of four possible types.
	\begin{enumerate}[label=(\roman*)]
		\item The edges $(u_i,v_j)$ that are both in $M_k$ and in $M_{k+d}$.
		\label{pairs1}
		\item Alternating paths of even lengths.
		\label{pairs2}
		\item Augmenting paths of $M_{k+d}$ with respect to $M_k$. 
		\label{pairs3}
		\item Augmenting paths of $M_k$ with respect to $M_{k+d}$.
		\label{pairs4}  
	\end{enumerate}
	For each alternating path $p$ of type \ref{pairs2} we have, by the maximality of the weights of the assignments,
	$
	w(p \cap M_k) = w(p \cap M_{k+d}).
	$
	For each pair pf paths $p_3$ of type \ref{pairs3} and  $p_4$ of type \ref{pairs4} we have by maximality, 
	$
	w((p_3 \cup p_4) \cap M_k) = w((p_3 \cup p_4) \cap M_{k+d}),
	$
	Hence, there are exactly $d$ more augmenting paths of type \ref{pairs3} than there are of type \ref{pairs4}, and if there exists a path of type \ref{pairs4} then there are two constants $c_3$ and $c_4$ such that $w(p_3 \cap M_{k+d})-w(p_3 \cap M_k)=c_3$ for every path $p_3$ of type \ref{pairs3}, and $w(p_4 \cap M_k)-w(p_4 \cap M_{k+d})=c_4$ for every path $p_4$ of type \ref{pairs4}.
	Thus, we can take alternatively for the assignment $M_{k+d}$ its original matched pairs from $d$ augmenting paths $p_1,\ldots,p_d$ of type \ref{pairs3}, and the other matches from $M_k$, so that the new assignment $M_{k+d}$ is of the same total weight as the original one.
	For each such $p_i$ of type \ref{pairs3}, let
	$$
	g_i = w(p_i \cap M_{k+d}) - w(p_i \cap M_k)
	$$
	be its induced gain.
	Suppose, without loss of generality, that 
	\begin{equation}
	g_1 \geq g_2 \geq \cdots \geq g_d.
	\label{eq:g_i}
	\end{equation}
	We construct the matchings $M_{k+i}$, $i=1,\ldots,d-1$, as
	$$
	M_{k+i} = M_{k+i-1} \triangle p_i
	$$
	and we have
	$$
	g_i = w(M_{k+i}) - w(M_{k+i-1}).
	$$
	Let
	\begin{equation}
	G = \sum_{i=1}^{d} g_i = w(M_{k+d}) - w(M_k).
	\label{eq:sum_g_i}
	\end{equation}
	Let $\WW_i$ be the total weight of a $(k+i)$-assignment, for $i=0,\ldots,d$, and let
	$$
	h_i = \WW_i - \WW_{i-1}, \quad \mbox{for } i=1,\ldots,d.
	$$
	Since $\fullchi{W} (x)$ is in FCF, $h_i \geq h_{i+1}$ for each $i$.
	The fact that the assignments $M_k$ and $M_{k+d}$ refer to adjacent terms of $\fullchi{W} (x)$ implies that the $(k+i)$-assignments, $i=1,\ldots,d-1$, refer to semi-essential terms and thus
	$$
	h_i = \frac{G}{d}, \quad \mbox{for } i=1,\ldots,d.
	$$
	By the maximality of the $(k+1)$-assignment,
	\begin{equation}
	g_1 \leq h_1 = \frac{G}{d},
	\label{eq:g_1}
	\end{equation}
	and it follows from \eqref{eq:g_i}, \eqref{eq:sum_g_i} and \eqref{eq:g_1} that
	$$
		g_1 = g_2 = \cdots = g_d = \frac{G}{d},
	$$
	and thus each matching $M_{k+i}$ is a $(k+i)$-assignment, $i=1,\ldots,d-1$.
	
	Since successive $k$-assignments that are computed by Algorithm~\ref{alg:essential} refer to adjacent terms of $\fullchi{W} (x)$ then by Theorem~\ref{th:adjacent} we can fill-in the gaps and efficiently compute all the missing $k$-assignments.
\end{proof}

\begin{algorithm}[H]
	\caption{Computing the semi-essential $k$-assignments}
	\begin{algorithmic}
		\STATE \textbf{Input:} $M_k, M_{k+d}$ \COMMENT{$k$- and $(k\!+\!d)$-assignments corresponding to adjacent terms of $\fullchi{W} (x)$} 
		\STATE \textbf{Output:} $out[i], i=1,\ldots,d\!-\!1$ \COMMENT{$(k\!+\!i)$-assignments $M_{k+i}$ that refer to semi-essential terms of $\fullchi{W}(x)$}
		\STATE \textbf{Initialization:} $out[0] \leftarrow M_k$
		\STATE Find $d-1$ disjoint augmenting paths $p_i$  of $M_{k+d}$ with respect to $M_k$ 
		\FOR{$i \leftarrow 1$ \TO  $d\!-\!1$}
			\STATE $out[i] \leftarrow out[i\!-\!1] \triangle p_{i}$
		\ENDFOR
	\end{algorithmic}
	\label{alg:semiessential}
\end{algorithm}

\subsection{Time complexity and conclusion}
A well-known algorithm that solves the assignment problem is the "Hungarian algorithm" (or "Hungarian method") of Kuhn \cite{Kuh55}, \cite{Kuh56}, based on the works of K\"{o}nig \cite{Kon16} and Egerv\'{a}ry \cite{Eger31}.
Munkers \cite{mun57} showed that the algorithm is of strongly polynomial time complexity, in fact $\mathcal{O}(n^4)$.
Later Karp \cite{EK72} and Tomizawa\cite{Tom72} gave a version of the algorithm whose running time is $\mathcal{O}(n^3)$.
The algorithm that is based on the successive short path algorithm (see \cite{AMO93}, Ch. 9.7 and Ch. 12.4) computes the sequence of $k$-assignments $k=1,\ldots,n$ in time $\mathcal{O}(n^3)$.

The algorithm of Gassner and Klinz \cite{GK10} for the parametric assignment problem is of time complexity $\mathcal{O}(n^3)$ (see also \cite{Hoo16}).
It is based on the algorithm of Young, Tarjan and Orlin \cite{YTO91} for the parametric shortest path problem.
In \cite{YTO91} the authors make use of the Fibonacci heap data structure of Fredman and Tarjan \cite{FT87} to improve the time complexity of the algorithm of Karp and Orlin \cite{KO81} for the parametric shortest path from $\mathcal{O}(n^3\log n)$ to $\mathcal{O}(n^3)$.

However, the G-K algorithm computes the assignments that refer to essential terms and possibly some of the semi-essential terms of the corresponding full characteristic maxpolynomial.
We showed here that, in fact, one can complement their algorithm and compute efficiently the missing assignments.
Given assignments $M_k$ and $M_{k+d}$, $d>1$, which refer to adjacent terms of the full characteristic maxpolynomial $\fullchi{W}(x)$, Algorithm~\ref{alg:semiessential} computes the sequence of $(k+i)$-assignments, $i=1,\ldots,d-1$, in time complexity $\mathcal{O}(dn)$: we just need to find $d-1$ augmenting paths of $M_{k+d}$ with respect to $M_k$ while assuming that each vertex $u_i$ points to its matched vertex $v_j$ and vice versa for both assignments.
Since the number of semi-essential terms of $\fullchi{W}(x)$ is bounded by $n$, the overall computation time of the semi-essential terms is $\mathcal{O}(n^2)$.

We conclude that the time complexity for computing all the $k$-assignments, $k=1,\ldots,n$, is $\mathcal{O}(n^3)$, but in special cases where there are many assignments that refer to semi-essential terms that are left to compute after completing the G-K algorithm then the fact that these extra assignments are then computed more efficiently can be a benefit. Such cases may occur when the range of the weight function is small and the pivot edges in the G-K algorithm are chosen in a way that the computed max-plus singular values are of large multiplicity. 
\\ \\ 
\noindent
\textbf{Acknowledgement.} 
\begin{small}
We thank Bettina Klinz for fruitful discussions.
This research was supported by the Austrian Science Fund (FWF) Project P29355-N35.
\end{small}

\bibliographystyle{alpha}
\bibliography{k_assignmets}

\begin{thebibliography}{DSDM02}

\bibitem[ABG01]{ABG01}
Marianne Akian, Ravindra Bapat, and St\'{e}phane Gaubert.
\newblock Generic asymptotics of eigenvalues using {M}in-{P}lus algebra.
\newblock In {\em In Proceedings of the Workshop on Max-Plus Algebras, IFAC
  SSSC’01}. Elsevier, 2001.

\bibitem[AMO93]{AMO93}
Ravindra~K. Ahuja, Thomas~L. Magnanti, and James~B. Orlin.
\newblock {\em Network flows}.
\newblock Prentice Hall, Inc., Englewood Cliffs, NJ, 1993.
\newblock Theory, algorithms, and applications.

\bibitem[BDM12]{BDM12}
Rainer Burkard, Mauro Dell'Amico, and Silvano Martello.
\newblock {\em Assignment problems (Revised Reprint)}.
\newblock Society for Industrial and Applied Mathematics (SIAM), 2012.

\bibitem[BJ16]{BJ16}
Ivan Belik and Kurt J\"ornsten.
\newblock A new semi-{L}agrangean relaxation for the {$k$}-cardinality
  assignment problem.
\newblock {\em J. Inf. Optim. Sci.}, 37(1):75--100, 2016.

\bibitem[But10]{But10}
Peter Butkovi{\v{c}}.
\newblock {\em Max-linear systems: theory and algorithms}.
\newblock Springer Monographs in Mathematics. Springer-Verlag London, Ltd.,
  London, 2010.

\bibitem[CG83]{CG83}
Ray~A. Cuninghame-Green.
\newblock The characteristic maxpolynomial of a matrix.
\newblock {\em J. Math. Anal. Appl.}, 95(1):110--116, 1983.

\bibitem[CGM80]{CGM80}
Ray~A. Cuninghame-Green and P.~F.~J. Meijer.
\newblock An algebra for piecewise-linear minimax problems.
\newblock {\em Discrete Appl. Math.}, 2(4):267--294, 1980.

\bibitem[DLM01]{DLM01}
Mauro Dell'Amico, Andrea Lodi, and Silvano Martello.
\newblock Efficient algorithms and codes for {$k$}-cardinality assignment
  problems.
\newblock In {\em Proceedings of the {F}irst {C}onference on {A}lgorithms and
  {E}xperiments {ALEX}98 ({T}rento)}, volume 110, pages 25--40, 2001.

\bibitem[DM97]{DM97}
Mauro Dell'Amico and Silvano Martello.
\newblock The {$k$}-cardinality assignment problem.
\newblock {\em Discrete Appl. Math.}, 76(1-3):103--121, 1997.
\newblock Second International Colloquium on Graphs and Optimization
  (Leukerbad, 1994).

\bibitem[DSDM02]{DeSDeM02}
Bart De~Schutter and Bart De~Moor.
\newblock The {QR} decomposition and the singular value decomposition in the
  symmetrized max-plus algebra revisited.
\newblock {\em SIAM Rev.}, 44(3):417--454, 2002.
\newblock Reprint of SIAM J. Matrix Anal. App. {{\bf{1}}9} (1998), no. 2,
  378--406 (electronic).

\bibitem[Ege31]{Eger31}
Jen{\H{o}} Egerv{\'a}ry.
\newblock {M}atrixok kombinatorius tulajdons{\'a}gair{\'o}l (in {H}ungarian)
  [{O}n combinatorial properties of matrices].
\newblock {\em Matematikai {\'e}s Fizikai Lapok}, 38:16--28, 1931.

\bibitem[EK72]{EK72}
Jack Edmonds and Richard~M. Karp.
\newblock Theoretical improvements in algorithmic efficiency for network flow
  problems.
\newblock {\em J. {ACM}}, 19(2):248--264, 1972.

\bibitem[FT87]{FT87}
Michael~L. Fredman and Robert~Endre Tarjan.
\newblock Fibonacci heaps and their uses in improved network optimization
  algorithms.
\newblock {\em J. Assoc. Comput. Mach.}, 34(3):596--615, 1987.

\bibitem[GK10]{GK10}
Elisabeth Gassner and Bettina Klinz.
\newblock A fast parametric assignment algorithm with applications in
  max-algebra.
\newblock {\em Networks}, 55(2):61--77, 2010.

\bibitem[Hoo15]{Hook14}
James Hook.
\newblock Max-plus singular values.
\newblock {\em Linear Algebra Appl.}, 486:419--442, 2015.

\bibitem[Hoo16]{Hoo16}
James Hook.
\newblock An algorithm for computing the eiginvalues of a max-plus matrix
  polynomial, 2016.
\newblock \\ preprint at http://eprints.maths.manchester.ac.uk/.

\bibitem[KO81]{KO81}
Richard~M. Karp and James~B. Orlin.
\newblock Parametric shortest path algorithms with an application to cyclic
  staffing.
\newblock {\em Discrete Appl. Math.}, 3(1):37--45, 1981.

\bibitem[K{\"{o}}n16]{Kon16}
D\'{e}nes K{\"{o}}nig.
\newblock {\"U}ber {G}raphen und ihre {A}nwendung auf {D}eterminantentheorie
  und {M}engenlehre.
\newblock {\em Math. Ann.}, 77(4):453--465, 1916.

\bibitem[Kuh55]{Kuh55}
Harold~W. Kuhn.
\newblock The {H}ungarian method for the assignment problem.
\newblock {\em Naval Res. Logist. Quart.}, 2:83--97, 1955.

\bibitem[Kuh56]{Kuh56}
Harold~W. Kuhn.
\newblock Variants of the {H}ungarian method for assignment problems.
\newblock {\em Naval Res. Logist. Quart.}, 3:253--258 (1957), 1956.

\bibitem[Mun57]{mun57}
James Munkres.
\newblock Algorithms for the assignment and transportation problems.
\newblock {\em J. Soc. Indust. Appl. Math.}, 5:32--38, 1957.

\bibitem[RLP19]{RLP19}
Amnon Rosenmann, Franz Lehner, and Aljo\v{s}a Peperko.
\newblock Polynomial convolutions in max-plus algebra.
\newblock {\em Linear Algebra Appl.}, 578:370--401, 2019.

\bibitem[Tom72]{Tom72}
Nobuaki Tomizawa.
\newblock On some techniques useful for solution of transportation network
  problems.
\newblock {\em Networks}, 1:173--194, 1971/72.

\bibitem[Vol04]{Vol04}
Anton Volgenant.
\newblock Solving the {$k$}-cardinality assignment problem by transformation.
\newblock {\em European J. Oper. Res.}, 157(2):322--331, 2004.

\bibitem[YTO91]{YTO91}
Neal~E. Young, Robert~E. Tarjan, and James~B. Orlin.
\newblock Faster parametric shortest path and minimum-balance algorithms.
\newblock {\em Networks}, 21(2):205--221, 1991.

\end{thebibliography}

\end{document}